\documentclass{amsart}
\usepackage{amsfonts,amsmath,amsthm,turnstile,color}
\usepackage{graphicx}
\usepackage{tikz}
\usepackage{pgf,tikz}
\usetikzlibrary{arrows}
\usepackage{color}
\usepackage[hypertexnames=false,linkcolor=black, colorlinks=true, citecolor=black, filecolor=black,urlcolor=black]{hyperref}

\theoremstyle{plain}
\newtheorem{Tw}{Theorem}[section]
\newtheorem{Wn}[Tw]{Corollary}
\newtheorem{Uw}[Tw]{Remark}
\newtheorem{lem}[Tw]{Lemma}

\theoremstyle{definition}
\newtheorem{obs}[Tw]{Observation}
\newtheorem{ex}[Tw]{Example}

\newcommand{\bR}{\mathbb{R}}
\newcommand{\bC}{\mathbb{C}}
\newcommand{\bN}{\mathbb{Z}_{\geq0}}
\newcommand{\bZ}{\mathbb{Z}}
\newcommand{\bK}{\mathbb{K}}
\newcommand{\scalar}[2]{\langle #1 , #2 \rangle}

\newcommand{\ini}{\operatorname{in}}
\newcommand{\fm}{\mathfrak m}
\newcommand{\gr}{\operatorname{gr}}
\newcommand{\cP}{\mathcal P}

\title{Loose edges and factorization theorems}

\author{Janusz Gwo\'zdziewicz}
\address{Janusz Gwo\'zdziewicz\\
	Institute of Mathematics\\
	Pedagogical University of Cracow\\
	Podchor\c a{\accent95 z}ych 2\\
	PL-30-084 Cracow, Poland}
\email{e-mail: gwozdziewicz@up.krakow.pl}

\author{Beata Hejmej}
\address{Beata Hejmej\\\\
	Institute of Mathematics\\
	Pedagogical University of Cracow\\
	Podchor\c a{\accent95 z}ych 2\\
	PL-30-084 Cracow, Poland}
\email{bhejmej1f@gmail.com}

\author{Bernd Schober}
\address{Bernd Schober\\
	Institut f\"ur Algebraische Geometrie\\
	Leibniz Universit\"at Hannover\\
	Welfengarten 1\\
	30167 Hannover\\
	Germany}
\email{schober@math.uni-hannover.de}

\subjclass[2010]{13F25, 12E05, 14B05}
\keywords{irreducibility, formal power series, Newton polyhedron}

\thanks{The third author is supported by the DFG-project ``Order zeta functions and resolutions of singularities" (DFG project number: 373111162)}

\begin{document}
	\maketitle
	\begin{abstract}
		Let $ R $ be a regular local ring with maximal ideal $ \fm $.
		We consider 
		elements $ f \in R $
		such that their Newton polyhedron has a loose edge. We show that if the symbolic restriction of 
		$f$ to such an edge is a product of two coprime polynomials, then $f$ 
		factorizes in the 
		$ \fm $-adic completion.
	\end{abstract}
	
	\section{Introduction}
	
	\textbf{Notation.}
	We denote by $\bR_{\geq0}$ (respectively $\bR_{>0}$) 
	the set of nonnegative (respectively positive) real numbers. 
	The symbol $\scalar{\cdot}{\cdot}$ denotes the standard scalar product.
	We use a multi-index notation 
	$\underline x^{\alpha}:=x_1^{\alpha_1}\cdots x_n^{\alpha_n}$ 
	for $\alpha=(\alpha_1,\dots,\alpha_n)$. 
	The ring of formal power series of variables $x_1$,\dots, $x_n$ 
	with coefficients in a field $\bK$ will be denoted by $\bK[\![x_1,\dots,x_n]\!]$. By $\bK[\underline{x}^{\pm}]$ we mean the ring $\bK[x_1,x_1^{-1},\dots,x_n,x_n^{-1}]$.

	\medskip
	Let $ R $ be a regular local ring with maximal ideal $ \fm $ and residue field $ \bK = R/\fm $.
	Let $ \widehat R $ be its $ \fm $-adic completion.
	(Recall that $ \widehat R = R $ if $ R $ is complete).
	For example, this is the case  $ R = \bK[\![x_1,\dots,x_n]\!] $.
	The main results of the article are factorization theorems for elements in $ R $,
	resp.~for polynomials with coefficients in $ R $, using convex geometry.
	
	Let $ ( \underline{x}) = (x_1, \ldots, x_n) $ be a regular system of parameters for $ R $.
	As shown in \cite[Proposition 2.1]{CPmixed} (applied for $ J = \{1, \ldots, n \} $), 
	since $ R $ is Noetherian and the map $ R \subset \widehat R $ is faithfully flat,
	every element $ f \in R $ can be expanded as a {\em finite} sum
	\[
	f = \sum_{ \alpha  \in \bN^n} a_\alpha \underline{x}^\alpha, 
	\]
	for $ a_\alpha \in R^\times \cup \{ 0 \} $. 
	Here, the symbol $ R^\times $ denotes the set of units in $ R $.

	Let $f=\sum_{\alpha\in \bN^n} a_{\alpha}\underline x^{\alpha} \in R $ be a nonzero element in $ R $.
	We define the {\em Newton polyhedron} $\Delta(f)$ of $f$ as the convex hull of the set 
	$\{ \alpha: a_{\alpha}\neq 0\} + \bR_{\geq 0}^n$. A set $\Delta\subset \bR_{\geq 0}^n$ is called a {\em Newton polyhedron} if $\Delta=\Delta(f)$ for some nonzero $f\in~R $.

	Given a Newton polyhedron $\Delta \subset \bR^n_{\geq 0} $, for any $\xi\in\bR_{\geq 0}^n$ we call 
	the set $$\Delta^{\xi} := \{\, a\in \Delta: \scalar{\xi}{a}=\min_{b\in \Delta} \scalar{\xi}{b}\,\}$$ 
	a {\em face} of $\Delta$. 
	A Newton polyhedron has a finite number of faces 
	since it is defined by a finite set of exponents.
	A face $\Delta^{\xi}$ is compact if and only if $\xi\in \bR_{>0}^n$. 
	A face of dimension 0  (respectively 1) is called  a {\em vertex} (respectively an {\em edge}).
	Following \cite[p.~123]{lipkovski1988newton}, we call a  compact edge of a Newton polyhedron a {\it loose edge} 
	if it is not contained in any compact face of dimension~$\geq 2$.

	Several Newton polyhedra are drawn in the pictures that follow. The segments marked in blue are loose edges.
	
	
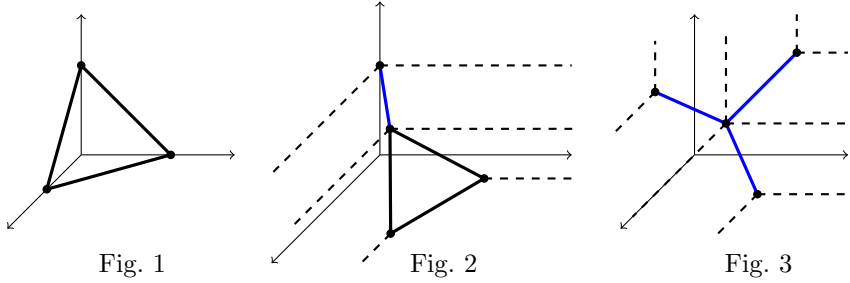
\begin{center}
		\begin{tikzpicture}[scale=1.7]
		\draw [->](0,0,0) -- (1.2,0,0); \draw[->](0,0,0) -- (0,1.1,0) ; 
		\draw[->](0,0,0) -- (0,0,1.5);
		\draw[very thick] (0,0.7,0) -- (0,0,0.7);
		\draw[very thick] (0.7,0,0) -- (0,0,0.7);
		\draw[very thick] (0,0.7,0) -- (0.7,0,0);
		\node[draw,circle,inner sep=1pt,fill=black] at (0,0.7,0) {};
		\node[draw,circle,inner sep=1pt,fill=black] at (0.7,0,0) {};
		\node[draw,circle,inner sep=1pt,fill=black] at (0,0,0.7) {};
		\node [below=1.2cm, align=flush center,text width=2cm] at (0.4,0,0)    {Fig.\ 1};
		\end{tikzpicture}
		\quad
		\begin{tikzpicture}[scale=1.7]
		\draw [->](0,0,0) -- (1.5,0,0); \draw[->](0,0,0) -- (0,1.2,0) ; 
		\draw[->](0,0,0) -- (0,0,2.2);
		\draw[thick, dashed] (0,0.7,0) -- (0,0.7,2.2);
		\draw[thick, dashed] (0,0.7,0) -- (1.5,0.7,0);
		\draw[very thick, color=blue] (0,0.7,0) -- (0.175,0.3,0.25);
		\draw[very thick] (0.175,0.3,0.25) -- (0.7,0,1.6);
		\draw[very thick] (0.175,0.3,0.25) -- (1,0,0.48);
		\draw[very thick] (0.7,0,1.6) -- (1,0,0.48);
		\draw[thick, dashed ] (0.7,0,1.6) -- (0.7,0,2.2);
		\draw[thick, dashed] (1,0,0.48) -- (1.7,0,0.48);
		\draw[thick, dashed] (0.175,0.3,0.25) -- (1.6,0.3,0.25);
		\draw[thick, dashed] (0.175,0.3,0.25) -- (0.175,0.3,2.2);
		\node[draw,circle,inner sep=1pt,fill=black] at (0,0.7,0) {};
		\node[draw,circle,inner sep=1pt,fill=black] at (0.175,0.3,0.25) {};
		\node[draw,circle,inner sep=1pt,fill=black] at (0.7,0,1.6) {};
		\node[draw,circle,inner sep=1pt,fill=black] at (1,0,0.48) {};
		\node [below=1.2cm, align=flush center,text width=2cm] at (0.5,0,0)    {Fig.\ 2};
		\end{tikzpicture} 
		\quad
		\begin{tikzpicture}[scale=1.7]
		\draw [->](0,0,0) -- (1.2,0,0); \draw[->](0,0,0) -- (0,1.1,0) ; 
		\draw[->](0,0,0) -- (0,0,1.5);
		\draw[very thick, color=blue] (0.8,0.8,0) -- (0.4,0.4,0.4);
		\draw[very thick, color=blue] (0.8,0,0.8) -- (0.4,0.4,0.4);
		\draw[very thick, color=blue] (0,0.8,0.8) -- (0.4,0.4,0.4);
		\draw[thick, dashed] (0.4,0.4,0.4) -- (0.4,1.1,0.4);
		\draw[thick, dashed] (0.4,0.4,0.4) -- (0.4,0.4,2.3);
		\draw[thick, dashed] (0.4,0.4,0.4) -- (1.35,0.4,0.4);
		\draw[thick, dashed] (0.8,0.8,0) -- (1.2,0.8,0);
		\draw[thick, dashed] (0.8,0.8,0) -- (0.8,1.1,0);
		\draw[thick, dashed] (0,0.8,0.8) -- (0,0.8,1.7);
		\draw[thick, dashed] (0,0.8,0.8) -- (0,1.2,0.8);
		\draw[thick, dashed] (0.8,0,0.8) -- (0.8,0,1.7);
		\draw[thick, dashed] (0.8,0,0.8) -- (1.5,0,0.8);
		\node[draw,circle,inner sep=1pt,fill=black] at (0.4,0.4,0.4) {};
		\node[draw,circle,inner sep=1pt,fill=black] at (0.8,0.8,0) {};
		\node[draw,circle,inner sep=1pt,fill=black] at (0.8,0,0.8) {};
		\node[draw,circle,inner sep=1pt,fill=black] at (0,0.8,0.8) {};
		\node [below=1.2cm, align=flush center,text width=2cm] at (0.5,0,0)    {Fig.\ 3};
		\end{tikzpicture}
	\end{center}
	
	The Newton polyhedron in Figure 1 does not have any loose edge. 
	This is the typical situation. 
	The Newton polyhedron in Figure 2 has a loose edge with the end point 
	at $(0,0,d)$. Weierstrass polynomials  $f\in \bK[\![x_1,\dots,x_n]\!][z]$ 
	such that $\Delta(f)$ is of this type were studied in \cite{artal2011},
	\cite{parusinski2012abhyankar}, 
	and  \cite{rond2017irreducibility}.

	In Figure 3 all compact edges are loose. A Newton polyhedron with this property 
	is called in \cite{perez2000singularites} a {\em polygonal Newton polyhedron}. 
	Notice that the term polygonal Newton polyhedron can be 
	a bit misleading since the union of compact edges in Figure 3 is not homeomomorphic 
	to any polygon. 
	
	Every compact edge of a plane Newton polyhedron is loose as illustrated in Figure 4. 
	
	\begin{center}
		\begin{tikzpicture}[
		scale = 0.6,
		foreground/.style = {  thick },
		background/.style = { dashed },
		line join=round, line cap=round
		]
		\draw[fill=black, opacity=0.1] (4,0)--(4.9,0)--(4.9,4.9)--(0,4.9)--(0,4)--(1,2)--(2,1)--cycle;
		\draw[foreground,->] (0,0)--+(5,0);
		\draw[foreground,->] (0,0)--+(0,5);
		\draw[very thick, color=blue] (1,2)--(2,1);
		\draw[very thick, color=blue] (0,4)--(1,2);
		\draw[very thick, color=blue] (2,1)--(4,0);
		\draw (1,-0.15) node[below] {$ $};
		\draw (-0.15,1) node[left] {$ $};
		\foreach \x in{} {
			\foreach \y in{}{
				\draw[fill, opacity=0.9]  (\x,\y) circle (0.5pt);
			}
		}; 
		\node[draw,circle,inner sep=1pt,fill=black] at (0,4) {};
		\node[draw,circle,inner sep=1pt,fill=black] at (1,2) {};
		\node[draw,circle,inner sep=1pt,fill=black] at (2,1) {};
		\node[draw,circle,inner sep=1pt,fill=black] at (4,0) {};  
		\node [below=0.4cm, align=flush center,text width=4cm] at (2.2,0)    {Fig.\ 4};
		\end{tikzpicture}
	\end{center}
	
	Consider $f=\sum_{\alpha\in \bN^n} a_{\alpha}\underline x^{\alpha} \in R $.
	The \textit{symbolic restriction} of $f$ to a compact face 
	$A := \Delta^\xi \subset \Delta(f)$,
	for some $ \xi \in \bR^n_{> 0} $, 
	is defined as the polynomial
	$$  f|_A=\sum_{\alpha\in A} \overline a_{\alpha} \underline X^{\alpha} \in \bK [X_1,\ldots, X_n ] , $$
	where $ \overline a_\alpha \in \bK $ 
	is the image of $ a_\alpha $ in the residue field of $ R $.
	The element $ f|_A $ is a polynomial since $ A $ being a compact face implies that $ A \cap \bN^n $ is finite.
	In section \ref{sec:general}, we provide a connection to a certain graded ring which explains the use of capital letters $ (X_1, \ldots, X_n) $.
	
	Below are the main results of the paper. 
	\begin{Tw} \label{t1}
		Let $ R $ be a regular local ring and let $ f \in R $ be a nonzero element in $ R $.	
		Assume that the Newton polyhedron $\Delta(f)$ has a  loose edge $E$.
		If~$f|_E$ is a product of two relatively prime polynomials $G$ and $H$, 
		where $G$ is not divided by any variable, then there exist 
		$ g$, $ h \in \widehat R $ in the completion of $ R $ such that
		$f= g  h $ in $ \widehat R $ and  $ g|_{E_1}=G$, $ h|_{E_2}=H$ for some $E_1$, $E_2$ such that 
		$ E $ is the Minkowski sum of $ E_1 $ and $ E_2 $,
		$E=E_1+E_2$. 
	\end{Tw}
	
	In the above theorem $E_1$ is a loose edge of $\Delta(g)$ parallel to $E$ and $E_2$ is a compact face of $\Delta(h)$ 
	which is a loose edge parallel to $E$ or a vertex 
	if $H$ is a monomial.

	\begin{Uw}
		The assumption of Theorem \ref{t1} that  $G$ is not divisible by any variable  cannot be omitted. 
		Consider the power series $f=(x_3^2+x_1x_2)(x_3+x_1 x_2)$. 
		Its Newton polyhedron has a loose edge $E$ with endpoints $(1,1,1)$, $(2,2,0)$ and 
		$f|_E=x_1x_2(x_3+x_1x_2)$. The irreducible factors of $f$ are $f_1=x_3^2+x_1x_2$ and 
		$f_2=x_3+x_1 x_2$. Hence $f$ cannot be a product of power series $g$, $h$ such that 
		$g|_{E_1}=x_2(x_3+x_1x_2)$ and $h|_{E_2}=x_1$.
	\end{Uw}
	
	\begin{Wn}\label{w1}
		Assume that the Newton polyhedron of $f\in R $ 
		has a loose edge and at least three vertices. 
		Then $f$ is not irreducible in the completion $ \widehat R $. 
	\end{Wn}
	
	\begin{Wn}\label{w2}
		Assume that the Newton polyhedron of $f\in R $ 
		has a loose edge $E$.
		If $f$ is irreducible in $ \widehat R $,
		then $E$ is the only compact edge of $\Delta(f)$ and $f|_E=cF^k$,
		where 
		$ c\in \bK^\times $ and
		$F \in \bK[X_1,\ldots, X_n]$ is an irreducible polynomial (with
		notations as
		before). 
		Moreover, if the residue field $\bK$ of $ R $ is algebraically closed, then there exist $a,b\in \bK, k\in \bN, \alpha,\beta\in\bN^n$  such that  $f|_E=(a \underline X^{\alpha}+b \underline X^{\beta})^k$ 
		with a primitive lattice vector $\alpha-\beta$.
	\end{Wn}
	
	We will say that a segment $E\subset\bR^{n+1}$ is \textit{descendant} if $E$ is parallel to some vector 
	$ c=(c_1,\dots,c_n, c_{n+1})$ such that $c_i\geq 0$ for $i=1,\dots,n$ and $c_{n+1}<0$.

	\begin{Tw} \label{t43}
		Let $ R $ be a regular local ring and let
		$f\in R[y] $ be a polynomial with coefficients in $ R $, for some independent variable $ y $. 
		Assume that the Newton polyhedron $\Delta(f) \subset \bR^{n+1}_{\geq 0} $ has a loose and  descendant edge $E$. 
		If $f|_E$ is a product of two relatively prime polynomials  $G, H \in \bK[X_1,\ldots, X_n][y] $, 
		where  $G$ is monic with respect to 
		$y$, then there exist
		$g,h\in \widehat R[y] $ 
		such that $f=gh$, 
		$g$ is a monic polynomial with respect to $ y $
		and  $g|_{E_1}=G$, $h|_{E_2}=H$ for some  $E_1$, $E_2$ such that $E=E_1+E_2$. 
		Moreover, if $ R = \bK[\![x_1,\dots,x_{n}]\!] $, then $ g $ and $ h $ are uniquely determined.
	\end{Tw}

	\medskip
	
	\section{Proofs  if $ R = \bK[\![x_1,\dots,x_{n}]\!] $}
	At the beginning of this section we establish some results about loose edges.

	\begin{lem}\label{Lem:edge}
		Let $\Delta$ be a Newton polyhedron with a loose edge $E$ that has ends 
		$a$, $b\in\mathbb{R}_{\geq 0}^n$.
		Then for every $c\in \Delta$ and every $\xi\in\mathbb{R}_{\geq0}^n$ such that 
		$\scalar{\xi}{a} = \scalar{\xi}{b}$ one has $\scalar{\xi}{c} \geq \scalar{\xi}{a}$.
	\end{lem}
	
	\begin{proof}
		Let $V$ be the  set of vertices of $\Delta$.  
		If $V=\{a,b\}$ then Lemma~\ref{Lem:edge} follows easily 
		since every point of $\Delta$ can be 
		written as $\lambda a+\mu b + z$ for some $z\in \bR_{\geq0}^n$ 
		and~$\lambda,\mu\geq 0$ such that $\lambda+\mu=1$.
		Hence in the rest of the proof assume that there exists a vertex of $\Delta$ different from $a$ or $b$ 
		and consider the function 
		$$\psi(\xi) =  \min_{v \in V\setminus\{a,b\}} \scalar{\xi}{v} - \scalar{\xi}{a},\quad
		\xi\in\mathbb{R}_{\geq0}^n.$$ 
		Since the set of vertices is finite, this function is well defined and continuous. 
		
		Since $E$ is a compact face of $\Delta$, there exists $\xi_0\in \mathbb{R}_{>0}^n$ such that 
		$\scalar{\xi_0}{a} = \scalar{\xi_0}{b}$ and 
		$\scalar{\xi_0}{c} > \scalar{\xi_0}{a}$  for all $c\in \Delta\setminus E$. 
		This yields $$\psi(\xi_0)>0. $$
		
		Suppose that there exist $c\in\Delta$ and $\xi_1\in\mathbb{R}_{\geq0}^n$ such that 
		$\scalar{\xi_1}{c} < \scalar{\xi_1}{a} = \scalar{\xi_1}{b}$.  
		Observe that  
		$c=\lambda_1v_1+\cdots+\lambda_s v_s+ z$, 
		for some $z\in\mathbb{R}_{\geq0}^n$, $v_1,\dots,v_s\in V$ and $\lambda_i\geq 0$ such that $\lambda_1+\cdots+\lambda_s=1$. If   $\scalar{\xi_1}{v} \geq \scalar{\xi_1}{a}$ for every vertex $v$ of $\Delta$, then   $\scalar{\xi_1}{c} =\lambda_1 \langle \xi_1,v_1\rangle + \cdots + \lambda_s\langle\xi_1, v_s\rangle +\langle \xi_1,z\rangle\geq (\lambda_1+\cdots+\lambda_s)\langle \xi_1,a\rangle =\langle \xi_1,a\rangle$, which contradicts the assumption $\scalar{\xi_1}{c} < \scalar{\xi_1}{a}$. Hence   $\scalar{\xi_1}{v} < \scalar{\xi_1}{a}$ for 
		at least one vertex $v\in V$. Thus $$\psi(\xi_1)<0.$$
		
		It follows from the above that there exist $\xi$ in the segment, joining $\xi_0$ and~$\xi_1$, 
		such that $\psi(\xi)=0$. 
		It means that there exists $v\in V\setminus\{a,b\}$ such that 
		$\scalar{\xi}{v}=\scalar{\xi}{a}=\scalar{\xi}{b}\leq \scalar{\xi}{d}$ for all $d\in\Delta$. 
		This implies that $\Delta^{\xi}$ is a compact (since $\xi\in\mathbb{R}_{>0}^n$) 
		face of dimension $\geq 2$, which contains $E$.
		Therefore, we obtained a contradiction to the hypothesis that $ E $ is a loose edge.
	\end{proof}
	
	\begin{lem}\label{Lem:only_vertices}
		Let $\Delta$ be a Newton polyhedron with a loose edge $E$ that has end points 
		$a=(a_1,\dots, a_n)$, $b=(b_1,\dots, b_n)$.
		If $\min(a_1,b_1)=\dots=\min(a_n,b_n)=0$, 
		then $a$ and $b$ are the only vertices of $\Delta$. \label{lc}
	\end{lem}
	
	\begin{proof}
		By assumption there exist two nonempty and disjoint subsets $A$, $B$ of the set of indices $\{1,\dots, n\}$ such that 
		$a_i>0$ for $i\in A$, $a_i=0$ for $i\in \{1,\dots, n\}\setminus A$,
		$b_j>0$ for $j\in B$ and $b_j=0$ for $j\in \{1,\dots, n\}\setminus B$.
		
		Let $c=(c_1,\dots,c_n)$  be an arbitrary point of $\Delta$ different from $a$ and $b$. 
		For any $i\in A$, $j\in B$ consider the vector $\xi_{ij}$, which has 
		only two nonzero entries: $1/a_i$ at place $i$ and  $1/b_j$ at place $j$
		($i\neq j$ since $A\cap B=\emptyset$). Then 
		$\scalar{\xi_{ij}}{a}=\scalar{\xi_{ij}}{b}=1$ and 
		$\scalar{\xi_{ij}}{c}=c_i/a_i+c_j/b_j$. 
		By Lemma~\ref{Lem:edge} we get $c_i/a_i+c_j/b_j \geq 1$.
		It follows that 
		$$ \min_{i\in A} c_i/a_i + \min_{j\in B} c_j/b_j \geq 1 .$$
		
		Choose constants  $\lambda\geq0$, $\mu\geq0$ such that 
		$\lambda\leq \min_{i\in A} c_i/a_i$, $\mu\leq \min_{j\in B} c_j/b_j$ 
		and $\lambda+\mu=1$.
		Then $c=\lambda a+\mu b+z$ for some $z\in \mathbb{R}_{\geq0}^n$, hence 
		$c$ cannot be a vertex of $\Delta$. 
	\end{proof}
	
	\begin{lem}\label{Lem:base}
		Let $c\in\bZ^n$ be a point with at least one positive and at least one negative coordinate. 
		Then there exists a basis $\xi_1$,\dots, $\xi_{n}$ of the vector space~$\bR^n$ such that 
		$\xi_i\in\bN^n$ for $i=1,\dots,n$ and $\scalar{\xi_i}{c}=0$ for $i=1,\dots,n-1$.
	\end{lem}
	
	\begin{proof}
		For any basis $\xi_1$, \dots, $\xi_{n}$ of~$\bR^n$ set 
		$w=(w_1,\dots,w_n):=(\scalar{\xi_1}{c},\dots,\scalar{\xi_n}{c})$.
		By the hypothesis of the lemma, for the standard basis 
		$\xi_1=(1,0,\dots,0)$, \dots, $\xi_n=(0,\dots,0,1)$ of $\bR^n$ there exist $i,j\in\{1,\dots,n\}$
		such that $w_i$, $w_j$ have opposite signs.
		We will gradually modify the standard basis  until we reach a basis such that only one 
		coordinate of $w$ is nonzero. 
		Below we outline the algorithm. 
		
		\begin{itemize}
			\item[1.] If there are only two nonzero entries $w_j$, $w_k$ of $w$ and $w_j+w_k=0$, then 
			replace $\xi_k$ by $\xi_k+\xi_j$ and stop.
			\item[2.] If there are two nonzero entries $w_j$, $w_k$ of $w$ of opposite signs such that 
			$|w_j|<|w_k|$, then replace $\xi_k$ by $\xi_k+\xi_j$ and go to step~1, otherwise go to step 3.
			\item[3.] Choose three nonzero entries  $w_j$, $w_k$, $w_l$ of $w$ such that 
			$w_k=w_l=-w_j$ (such a choice is possible since in this step all nonzero entries 
			of $w$ have the same absolute value), replace $\xi_k$ by $\xi_k+\xi_j$ and go to step~1.
		\end{itemize}

		After every loop of the above algorithm $w_k$ is replaced by $w_k+w_j$. Hence 
		the number $\sum_{i=1}^n |w_i|$ decreases and the algorithm must terminate.
	\end{proof}
	
	We encourage the reader to apply the algorithm from the proof in a simple case, 
	for example for $c=(2,3,-4)$. 
	
	\medskip 
	From now on up to the end of this section we fix a loose edge $E$. 
	Let $c$ be a primitive lattice vector parallel to $E$.
	Applying Lemma~\ref{Lem:base} to $c$ we find ${n-1}$~linearly independent vectors
	$\xi_1$,\dots, $\xi_{n-1}\in \bN^n$ which are orthogonal to $E$.
	For every monomial $\underline{x}^{\alpha}$ we set 
	$\omega(\underline{x}^{\alpha}) := 
	(\scalar{\xi_1}{\alpha},\,\dots,\scalar{\xi_{n-1}}{\alpha})$. 
	We call this vector a {\it weight} of $\underline{x}^{\alpha}$. 
	Since 
	$\omega(\underline{x}^{\alpha}\underline{x}^{\beta}) =
	\omega(\underline{x}^{\alpha})+\omega(\underline{x}^{\beta})$ 
	for any monomials $\underline{x}^{\alpha}$, $\underline{x}^{\beta}$, 
	the ring $\mathbb{K}[x_1,\dots,x_n]$ turns into a graded ring 
	$\bigoplus_{w\in\bN^{n-1}} R_{w}$,
	where $R_{w}$ is spanned by monomials of weight $w$. 
	We have that $ R_{(\underline 0 )} = \bK$. Indeed, if $\underline{x}^{\alpha}\in R_{(\underline 0 )}  $ for some $\alpha\ne 0$ and $\beta$ is a point of $E$, then the points $\beta + t \alpha$, for $t>0$, belong to $E$, hence $E$ is unbounded.
	
	All monomials of $R_w$ are of the form $\underline{x}^{a+ic}$ where 
	$\underline{x}^{a}$ is a fixed monomial of $R_w$ and $i$ is an integer. This shows that $R_w$ is a finite-dimensional vector space over $\mathbb{K}$
	since there is only a finite number of integers $i$ such that all coordinates of $a+ic$ are nonnegative. 
	We denote by $ \dim R_w $ the dimension of $ R_w $ as a vector space over $ \bK $.
	It may happen that $\dim R_w=0$. 
	Take for example the weight $\omega(x_1^{\alpha_1}x_2^{\alpha_2})=3\alpha_1+2\alpha_2$.
	Then there is no  monomial of weight 1, hence $R_1\subset \mathbb{K}[x_1,x_2]$ is a zero-dimensional vector space. 
	
	Let $M\subset \bN^{n-1}$ be the set of weights satisfying the condition: 
	$z\in M$ if and only if there exists $\underline{x}^{\alpha}\in \bK[\underline{x}^{\pm}]$ 
	such that $\omega(\underline{x}^{\alpha})=z$ and $\scalar{\xi}{\alpha}\geq 0$ for all 
	$\xi\in\mathbb{R}_{\geq0}^n$ which are orthogonal to $E$. 
	Observe that $M$ is closed under addition. Moreover for any  $w\in \bN^{n-1}$ such that $\dim R_w>0$ we have $w\in M$.

	\begin{lem}\label{Lem:sum}
		Let $w\in\bN^{n-1}$ and $z\in M$. 
		Assume that~$R_w$ contains two coprime monomials. 
		Then $$\dim R_{w+z}=\dim R_w+ \dim R_z-1.$$\label{l1}
	\end{lem}
	
	\begin{proof}
		For any $v\in\bN^{n-1}$ 
		the dimension of the vector space $R_v$ is equal to the number of 
		monomials ${\underline x}^{\alpha}$ such that $\omega({\underline x}^{\alpha})=v$, hence 
		is equal to the number of lattice points, i.e. all points with integer coordinates, in the set 
		$$l_v = \{\,\alpha\in \bR_{\geq 0}^n : 
		(\scalar{\xi_1}{\alpha},\dots, \scalar{\xi_{n-1}}{\alpha})=v\,\} . $$
		Notice that $l_v$ is the intersection of the straight line 
		$\{a+tc: t\in \mathbb{R}\}$ with $\bR_{\geq 0}^n$, where $\omega({\underline x}^{a})=v$ and 
		$c$ is a primitive lattice vector parallel to the edge $E$.

		By the assumption, $R_w$ contains two coprime monomials 
		${\underline x}^{a}$ and ${\underline x}^{b}$. 
		Hence $\min(a_1,b_1)=\dots=\min(a_n,b_n)=0$ which yields the partition 
		of $\{1,\dots,n\}$ to three sets 
		$A=\{i\in\{1,\dots,n\}: a_i=0, b_i>0\}$, 
		$B=\{i\in\{1,\dots,n\}: a_i>0, b_i=0\}$ and
		$C=\{i\in\{1,\dots,n\}: a_i=0, b_i=0\}$. 
		Since $A$ and $B$ are nonempty,  $a$ and $b$ 
		are the endpoints of the segment $l_w$. 
		We may assume without loss of generality that the vector $c$ is pointed 
		in the direction of $b-a$. If  $\dim R_w=r+1$, then 
		the lattice points  of $l_w$ are $a$, $a+c$, \dots, $b=a+rc$.
		
		If the lattice points of $l_z$  are 
		$d$, $d+c$, \dots, $d+s c$ then
		$a+d$, $a+d+c$, \dots, $a+d+(r+s)c$ are the lattice points of $l_{w+z}$
		(see Figure 5). It is easy to check that the points $a+d-c$ and $a+d+(r+s+1)c$ do not belong to $\bR_{\geq 0}^n$. Hence $\dim R_{w+z}=r+s$.  This ends the proof in the case $\dim R_z>0$. 
		
		\begin{center}
			\begin{tikzpicture}[
			scale = 1,
			foreground/.style = {  thick },
			background/.style = { dashed },
			line join=round, line cap=round
			]
			\draw[foreground,->] (0,0)--+(6,0);
			\draw[foreground,->] (0,0)--+(0,3.5);
			\draw[thick] (5,0)--(4.5,0.25);
			\draw[dashed] (4.5,0.25)--(1.25,1.875);
			\draw[thick] (1.25,1.875)--(0,2.5);
			\draw (1,-0.15) node[below] {$ $};
			\draw (-0.15,1) node[left] {$ $};
			\node[draw,circle,inner sep=1pt,fill=black] at (0,2.5) {};
			\node[draw,circle,inner sep=1pt,fill=black] at (1,2) {};
			\node[draw,circle,inner sep=1pt,fill=black] at (5,0) {};
			\node at (0.3,2.6) {$a$}; 
			\node at (1.5,2.1) {$a+c$}; 
			\node at (5.4,0.2) {$a+rc$};   
			
			[
			scale = 1,
			foreground/.style = {  thick },
			background/.style = { dashed },
			line join=round, line cap=round
			]
			\draw[foreground,->] (6.5,0)--+(4.5,0);
			\draw[foreground,->] (6.5,0)--+(0,2.75);
			\draw[thick] (10,0)--(9.5,0.25);
			\draw[dashed] (9.5,0.25)--(6.5,1.75);
			\draw[thick] (6.5,1.75)--(8.25,0.875);
			\draw (1,-0.15) node[below] {$ $};
			\draw (-0.15,1) node[left] {$ $};
			\foreach \x in{} {
				\foreach \y in{}{
					\draw[fill, opacity=0.9]  (\x,\y) circle (0.5pt);
				}
			}
			\node[draw,circle,inner sep=1pt,fill=black] at (7,1.5) {};
			\node[draw,circle,inner sep=1pt,fill=black] at (8,1) {};
			\node[draw,circle,inner sep=1pt,fill=black] at (9.7,0.15) {};
			\node at (7.3,1.7) {$d$};
			\node at (8.5,1.2) {$d+c$};
			\node at (10.3,0.3) {$d+sc$}; 
			
			\end{tikzpicture}
		\end{center}
		
		\begin{center}
			\begin{tikzpicture}
			[
			scale = 1,
			foreground/.style = {  thick },
			background/.style = { dashed },
			line join=round, line cap=round
			]
			\draw[foreground,->] (0,0)--+(10,0);
			\draw[foreground,->] (0,0)--+(0,5);
			\draw[thick] (0,4.25)--(2.75,2.875);
			\draw[dashed] (2.75,2.875)--(8,0.25);
			\draw[thick] (8.5,0)--(8,0.25);
			\draw (1,-0.15) node[below] {$ $};
			\draw (-0.15,1) node[left] {$ $};
			\node[draw,circle,inner sep=1pt,fill=black] at (0.5,4) {};
			\node[draw,circle,inner sep=1pt,fill=black] at (1.5,3.5) {};
			\node[draw,circle,inner sep=1pt,fill=black] at (2.5,3) {};
			\node[draw,circle,inner sep=1pt,fill=black] at (8.2,0.15) {};
			\node at (1,4.2) {$a+d$}; 
			\node at (2.3,3.7) {$a+d+c$};
			\node at (3.35,3.2) {$a+d+2c$};
			\node at (9.5,0.34) {$a+d+(r+s)c$};
			\node [below=0.6cm, align=flush center,text width=6cm] at (4.5,0)    {Fig.\ 5};
			\end{tikzpicture}
		\end{center}
		
		Now suppose that $\dim R_z=0$. Let ${\underline x}^{d}\in\bK[\underline{x}^{\pm}]$ be any monomial  
		with integer exponents such that  $\omega({\underline x}^{d})=z$.
		Replacing this monomial by ${\underline x}^{d+kc}$, for suitably chosen integer $k$, 
		we may assume that  $d_{i_0}<0$ for some  $i_0\in A$ and 
		$d_i+c_i\geq 0$  for all $i\in A$. 
		
		Take the vector $v_j=\frac{1}{c_{i_0}}e_{i_0}-\frac{1}{c_j}e_j$, where  $e_1,\dots$, $e_n$ is the standard basis of $\bR^n$  and $j\in B$. Then $v_j\in \bR_{\geq 0}^n$ is orthogonal to $E$. By the assumption that $z\in M$ 
		(use $ \xi = v_j $ and $ \alpha = d $ in the definition of $ M $)
		we get the inequality
		$d_{i_0}/c_{i_0}-d_j/c_j  = \scalar{v_j}{d} \geq 0$. Hence $d_j> 0$ for all $j\in B$,
		since, by the definitions of $ A, B $ and the convention that $ c $ points into the direction of $ b - a $, we have $ c_j < 0 $ for $ j \in B $ and $ c_{i_0} > 0 $ for $ i_0 \in A $, and since $ d_{i_0} <0 $, by assumption. 
		
		If $j\in C$ then $e_j$ is orthogonal to $E$. 
		By the assumption that $z\in M$ 
		(use $ \xi = e_j $ and $ \alpha = d $ in the definition of $M$)
		we get the inequality $d_j\geq 0$ 
		for all $j\in C$.
		
		Notice that $d_{j}+c_j<0$ for at least one $j\in B$, otherwise all exponents of ${\underline x}^{d+c}$ 
		would be nonnegative,
		which would contradict $ \dim R_z = 0 $. 
		
		The lattice points of $l_{w+z}$ are of the form $a+d+ic$, where $i\in \bZ$. 
		Since $d_{i_0}<0$, the point $a+d$ has a negative coordinate. 
		Since $d_{j}+c_j<0$ for some $j\in B$, the point $a+d+(r+1)c=b+d+c$ has a negative coordinate. 
		We claim that 
		all coordinates of the points $a+d+ic$ are nonnegative, for $i=1,\dots,r$:
			If $ j \in A $, then we have seen that $ d_j + c_j \geq 0 $ and $ c_j > 0 $,
			and hence, $ a_j + d_j + i c_j = a_j + (c_j + d_j) + (i-1) c_j \geq 0 $, for $ i \geq 1$.
			On the other hand, if $ j \in C $ then $ a_j = b_j = 0 $, thus $ c_j = 0 $, and we have seen that $ d_j \geq 0 $.
			This provides $ a_j + d_j + i c_j =  d_j \geq 0 $, for $ j \in C  $.
			Finally, for $ j \in B $, we have $ d_j > 0 $ and $ b_j = 0 $, and since $ c = (b-a)/r$,
			we have that $ a_j + d_j + i c_j =  a_j + d_j - i a_j/r = (1 - i/r) a_j + d_j > 0 $, for $ i = 1, \ldots, r $.
			This proves the claim.
		Hence $\dim R_{z+w}=r$ which finishes the proof. 
	\end{proof}

	In the appendix, we outline an alternative proof for Lemma~\ref{Lem:sum}.
	
	\begin{lem} \label{l3}
		Let $G\in R_w$ and $H\in R_z$ be coprime polynomials. 
		If $G$ is not divisible by any monomial  then for every $i\in M$ 
		$$ G R_{z+i} + H R_{w+i} = R_{w+z+i}.
		$$
	\end{lem}
	
	\begin{proof}
		Consider the following  sequence 
		$$ 0\longrightarrow{} R_i\stackrel{\Phi}{\longrightarrow} 
		R_{z+i}\times R_{w+i} \stackrel{\Psi}{\longrightarrow} R_{w+z+i},
		$$ 
		where 
		$\Phi\colon \eta \mapsto (\eta H,-\eta G)$ and $\Psi\colon (\psi,\varphi)\mapsto \psi G+\varphi H$. Take any $(\psi,\varphi)\in R_{z+i}\times R_{w+i}$ such  that $\Psi(\psi,\varphi)=\psi G+ \varphi H=0$. Since $G$ and $H$ are coprime, there exists $\eta \in R_i$ such that $\psi=\eta H$ and $\varphi = - \eta G$. Thus $\ker \Psi\subset {\rm im}\, \Phi$. The opposite inclusion is obvious. 
		Moreover, it is easy to see that $\Phi$ is injective. 
		It follows that the above sequence is exact.

		By assumption $ G \in R_w \subset \bK [x_1, \ldots, x_n] $ is neither a constant nor a monomial.
		Hence there appear at least two monomials in $ G $ with non-zero coefficient
		and they can be chosen to be coprime as $ G $ is not divisible by a monomial.
		Since both monomials are contained in $ R_w $, we obtain
		that $R_w$ satisfies the hypothesis of Lemma~\ref{l1}.
		Hence we get 
		$\dim {\rm Im}\,\Psi = \dim R_{z+i}+\dim R_{w+i} - \dim R_i = 
		\dim R_{z+i}+ (\dim R_w+\dim R_i - 1) - \dim R_i = 
		\dim R_w+\dim R_{z+i} -1 = \dim R_{w+z+i}$, 
		which implies that $\Psi$ is surjective.
	\end{proof}
	
	\smallskip 
	
	\begin{proof}[Proof of Theorem \ref{t1}]
		Suppose that $ R =  \bK[\![x_1,\dots,x_n]\!] $.
		Then, $ \widehat R = R $ and $ f \in \bK[\![\underline x]\!] $ can be written as $ f = \sum_{ \alpha  \in \bN^n} a_\alpha \underline{x}^\alpha $, with coefficients $ a_\alpha \in \bK $.

		Since all monomials of  $f|_E$ have the same weight, the polynomial $f|_E$ is homogeneous in a graded ring $\mathbb{K}[x_1,\dots,x_n]$.
		Then by an elementary property of graded rings the equality $f|_E=GH$ implies that  $G\in R_{w}$ and $H\in R_{z}$ 
		for some $w$, $z\in \bN^{n-1}$. 
		Let ${\underline x}^{\alpha}$ be any monomial which appears 
		with nonzero coefficient in the power series~$f$ 
		and ${\underline x}^{\alpha_0}$ be a fixed monomial of $f|_E$. 
		By Lemma~\ref{Lem:edge} we have $\scalar{\xi}{\alpha} \geq \scalar{\xi}{\alpha_0}$
		for every $\xi\in\mathbb{R}_{\geq0}^n$ which is orthogonal to $E$. 
		This yields $\omega({\underline x}^{\alpha-\alpha_0})\in M$. 
		Since $\omega({\underline x}^{\alpha_0})=w+z$, 
		we get ${\underline x}^{\alpha}\in R_{w+z+i}$ for some $i\in M$ 
		and hence $f=\sum_{i\in M} f_{w+z+i}$, where $f_{w+z+i}\in R_{w+z+i}$.

		Let $g_w:=G$ and $h_{z}:=H$.  Then $f_{w+z}=g_{w}h_{z}$. 
		Let us set  $M$ in degree-lexicographic order.
		Using Lemma~\ref{l3} we can find recursively 
		$g_{w+i}\in R_{w+i}$ and $h_{z+i}\in R_{z+i}$ 
		such that 
		$$g_{w}h_{z+i}+h_{z}g_{w+i}=f_{w+z+i}-F_i,
		$$ 
		where 
		$$F_i=\sum_{\substack{k+l=i,\\k,l<i}}g_{w+k}h_{z+l}.$$ 
		If $g:=\sum_{i\in M} g_{w+i}$ 
		and $h:=\sum_{i\in M} h_{z+i}$, then $f=gh$.
		
		Let $\xi=\xi_1+\dots+\xi_{n-1}$,
		where $ \xi_1, \ldots, \xi_{n-1}, \xi_n \in \bN^n $ is the basis of $ \bR^n $ obtained in Lemma~\ref{Lem:base}.
		Then for  $E_1:=\Delta(g)^{\xi}$ and 
		$E_2:=\Delta(h)^{\xi}$ we have $g|_{E_1}=g_{w}$, $h|_{E_2}=h_{z}$ 
		and $E=E_1+E_2$. 
	\end{proof}

	\begin{proof}[Proof of Corollary \ref{w1}]
		Assume that $a=(a_1,\dots, a_n)$, $b=(b_1,\dots,b_n)$ 
		are the ends of a loose edge $E$ of the Newton polyhedron $\Delta(f)$. 
		Since $\Delta(f)$ has at least three vertices,  Lemma \ref{lc} implies that 
		$c=(\min(a_1,b_1),\dots,\min(a_n,b_n))$ has at least one nonzero coordinate.
		The monomials $\underline{x}^{a}$ and $\underline{x}^{b}$ 
		appear in the polynomial $f|_E$ with nonzero coefficients and their greatest common divisor equals
		$\underline{x}^{c}$.
		Thus $ G := \underline{x}^{-c}\cdot f|_E$ is not divisible  by any variable, 
		so $ G $ 
		and $ H := \underline{x}^c$ are relatively prime.
		Using Theorem \ref{t1} we obtain that $f$ is not irreducible.
	\end{proof}
	
	\begin{lem} \label{l4}
		Let $G\in R_w$ and $H_j\in R_{z_j}$ for $j=1,2$. 
		Assume that for every $i\in M$ 
		$$ G R_{z_j+i} + H_j R_{w+i} = R_{w+z_j+i}
		$$
		for $j=1,2$. Then for every $i\in M$ 
		$$ G R_{z_1+z_2+i} + H_1H_2 R_{w+i} = R_{w+z_1+z_2+i}.
		$$
	\end{lem}
	
	\begin{proof} By assumptions of the lemma we get
	\begin{eqnarray*}
			 G R_{z_1+z_2+i} + H_1H_2 R_{w+i} &= &
			G (R_{z_1+z_2+i}+H_1 R_{z_2+i}) + H_1H_2 R_{w+i}  \\
			&=&  G R_{z_1+z_2+i} + H_1 ( G R_{z_2+i} + H_2 R_{w+i} ) \\ 
			&=& G R_{z_1+z_2+i} + H_1R_{w+z_2+i} = R_{w+z_1+z_2+i} \,.
			\hspace{30pt}\qedhere  
		\end{eqnarray*}
	\end{proof}

	Let us consider the case that $f\in R[y] $ is a polynomial with coefficients in $ R $, 
	for some independent variable $ y $.
	Assume that $E \subset \bR^{n+1}_{\geq 0} $ is a descendant loose edge of $ \Delta(f)  $. 
	As before, we have $\bK[x_1,\dots,x_n][y] \cong \bigoplus_{w\in\bN^{n}} R_{w}$ 
	(with the analogous notations).
	By definition, there exists  a lattice vector  
	$c=(c_1,\dots,c_n, c_{n+1})$ parallel to $E$ such that $c_i\geq 0$ for $i=1,\dots,n$ and $c_{n+1}<0$.
	Let $R_w':=R_w\cap \bK[x_1,\dots,x_n]$ for $w\in\bN^{n}$. Since every line parallel to
	$E$ intersects $\bR^{n}\times \{0\}$ transversely, the dimension of  $R_w'$ is either 0 or 1. 
	
	We claim that for every $w\in M$, $z\in \bN^{n}$, $0\neq H\in R_z'$ and $\psi\in R_{w+z}'$  
	one has $\psi/H\in R_w'$.
	To prove this claim it is enough to consider $H=\underline{x}^{\alpha}\in  R_z'$ 
	and $\psi=\underline{x}^{\beta}\in R_{w+z}'$. Then $\omega(\underline{x}^{\beta-\alpha})=w$.
	Denote $v_i:=e_i-\frac{c_i}{c_{n+1}}e_{n+1}\in\mathbb{R}_{\geq 0}^{n+1}$, where $e_1,\dots,e_n, e_{n+1}$ is the standard basis of $\mathbb{R}^{n+1}$. 
	Since every vector $v_i$ is orthogonal to  $E$ and $w\in M$, we have $0\leq\langle v_i,\beta-\alpha\rangle = \beta_i-\alpha_i$ for $i=1,\dots,n$. 
	Thus  $\psi/H=\underline{x}^{\beta-\alpha}$ is a monomial with nonnegative exponents. 
	
	\begin{lem} \label{l5}
		Let $G\in R_w$ and $H\in R_z$ be coprime polynomials. 
		If $G$ is monic with respect to $y$,
		then for every $i\in M$ 
		\begin{equation} \label{eq2}
		G R_{z+i} + H R_{w+i} = R_{w+z+i}.
		\end{equation}
	\end{lem}
	
	\begin{proof} First, we prove~(\ref{eq2}) in the special case
		$G=y$ and $H\in R_z'$. 
		
		Every $F\in R_{w+z+i}$ can be written in the form 
		$F=y \phi + \psi$ where $\phi, \psi$ are polynomials and 
		$\psi$ does not depend on $y$. Then 
		$\phi\in R_{z+i}$ and $\psi=H\psi'$ where $\psi' = \psi/H\in R_{w+i}'$.
		This gives~(\ref{eq2}) in the special case. All remaining cases follow from Lemma~\ref{l3} and Lemma~\ref{l4}.
	\end{proof}
	
	\smallskip 
	
	\begin{proof}[Proof of Theorem \ref{t43}]
		
		Suppose that $ R =  \bK[\![x_1,\dots,x_n]\!] $.
		Then, $ \widehat R = R $ and any element in $ R[\![y]\!] = \bK[\![\underline x, y ]\!] $ 
		has an expansion with coefficients contained in the residue field $ \bK $.
		
		Proceeding as in the proof of Theorem \ref{t1}, but using Lemma~\ref{l5} instead of Lemma~\ref{l3}
		we obtain $\overline{g},\overline{h}\in\bK[\![\underline x, y ]\!]$ such that  
		$f=\overline{g}\overline{h}$, $\overline{g}|_{E_1}=G$ and $\overline{h}|_{E_2}=H$ 
		for some segments $E_1,E_2$, where $E=E_1+E_2$. 
		The assumptions that the loose edge $E$ is descendant and the polynomial $G$  
		is monic with respect to  $y$ imply that the Newton polyhedron of $\overline{g}$ 
		has a~vertex $(0,\dots,0,d)$ for some positive integer $d$.
		Therefore $\overline{g}(0,\dots,0,y)=v y^d$ for some $v\in\bK[\![y]\!]$ such that $v(0)\ne 0$. 
		It means that $\overline{g}$ fulfills the assumptions of the Weierstrass Preparation Theorem, which implies that $\overline{g}=u \hat{g}$, 
		where $u$ is a power series such that $u(0)\ne 0$
		and $\hat{g}\in \bK[\![x_1,\dots,x_{n}]\!][y]$  is a Weierstrass polynomial of degree $d$. 
		Putting $g=\hat{g}$ and $h=u\overline{h}$ we get $f=gh$.  
		Since we can also obtain $h$ using the polynomial division of $f$ by $g$
		in the ring $\bK[\![x_1,\dots,x_{n}]\!][y]$,
		we conclude that $h$ is a polynomial with respect  to~$y$. 
	\end{proof}

	\smallskip

	\section{Extending to regular local rings}
	\label{sec:general}

	After discussing the case of a formal power series ring over a field $ \bK $, 
	we explain how the results can be extended to the completion of a regular local ring
	by using graded rings. 
	
	Let $ ( R, \fm, \bK = R/\fm ) $ be a regular local ring,
	let $ ( \underline{x}) = (x_1, \ldots, x_n) $ be a regular system of parameters,
	and let $ \widehat R $ be the $ \fm $-adic completion of $ R $.
	We consider $ f \in R $ such that its Newton polyhedron has a loose edge $ E $.
	
	\smallskip 
	
	First, let us point out that 
	Lemmas \ref{Lem:edge}, \ref{Lem:only_vertices}, \ref{Lem:base} are results on convex geometry.
	In fact, they are true for any $ F $-subset $ \Delta \subset \bR^n_{\geq 0}  $, which is closed convex subset such that $ \Delta + \bR^n_{\geq 0} = \Delta $, \cite[p.~260]{HiroCharPoly}.
	In particular, they are independent of $ R $. 
	Furthermore, the proof of Corollary \ref{w1} only uses Theorem~\ref{t1} and Lemma~\ref{Lem:only_vertices}
	and thus can be applied word by word.
	
	\begin{obs}
		Recall that, after Lemma \ref{Lem:base},
		we introduced a weight given by
		$\omega(\underline{x}^{\alpha}) := 
		(\scalar{\xi_1}{\alpha},\,\dots,\scalar{\xi_{n-1}}{\alpha})$,
		where $\xi_1$,\dots, $\xi_{n-1}\in \bN^n$ are linearly independent vectors
		which are orthogonal to $E$. 
		The map $ \omega $ induces a monomial valuation (denoted also by $\omega$) on $ R $
		with values in $ \bN^{n-1}$, i.e. for any nonzero $ h = \sum_{ \alpha  \in \bN^n} b_\alpha \underline{x}^\alpha \in R $, where  $ b_\alpha \in R^\times \cup \{ 0 \} $, 
		we set  
		$$ \omega(h) := \inf \{ \omega (\underline{x}^{\alpha}) : b_\alpha \neq 0 \} .$$
		
		We equip $ \bN^{n-1} $ with the partial ordering defined by 
		$ \alpha \leq \beta   \Leftrightarrow \beta = \alpha + \gamma $, for some $ \gamma \in \bN^{n-1} $. 
		The {\em graded ring of $ R $ associated to $ \omega $} is defined as
		\[
		\gr_\omega (R) := \bigoplus_{ w \in \bN^{n-1} }
		\cP_w / \cP_w^+, 
		\]
		where $ \cP_w := \{ h \in R : \omega(h) \geq w \} $ and 
		$ \cP_w^+ := \{ h \in R : \omega(h)  > w \} $.
		Note that $ \cP_{(\underline 0 ) } = R $ and $ \cP_{(\underline 0 ) }^+ = \fm $, since $ \omega(x_i) > ( \underline 0 ) $ for all $ i = 1, \ldots, n  $.
		Hence, we have $ \cP_{(\underline 0 ) } / \cP_{(\underline 0 ) }^+ = \bK $. Denote by $X_i\in \gr_\omega(R)$ the image of $x_i\in R$ in the graded ring.
		Then the {\em initial form of $ h $ with respect to $ \omega $} is defined by $ \ini_\omega(h) := X_i $ if $h=x_i$ and otherwise,
		$ \ini_\omega(h) :=  \sum_{ \alpha  : \omega(\underline{x}^\alpha) = \omega(h) } \overline b_\alpha \underline{X}^\alpha $,
		where $ \overline b_\alpha \in \bK $ is the image of $ b_\alpha $ in the residue field. Since $ \omega $ is a monomial valuation, we have $ \cP_w / \cP_w^+ = R_w $, for all $ w\in \bN^{n-1} $, i.e.,
		$ \gr_\omega (R) = \bigoplus_{ w \in \bN^{n-1} }
		R_w \cong \bK [ X_1, \ldots, X_n ] $.

		In conclusion, we see that the second set of ingredients for the proofs of Theorems \ref{t1} and \ref{t43}, Lemmas \ref{Lem:sum} and \ref{l3}, are results on the graded ring $ \gr_\omega(R) $
		without the restriction $ R = \bK[\![\underline x]\!] $.
	\end{obs} 	
	
	The essential part for extending the proof to any regular local ring, is the following construction 
	to find a {\em lift} $ g \in R $ for a given homogeneous element $ G \in R_w \subset \gr_\omega (R) \cong \bK[X_1, \ldots, X_n] $, where $ w \in \bN^{n-1} $: 		
	We can write $ G $ as a finite sum
	$ G = \sum_{\alpha \in \bN^n} \lambda_\alpha X^\alpha $,
	for $ \lambda_\alpha \in \bK $.
	For every $ \alpha \in \bN^n $ with $ \lambda_\alpha \neq 0 $, we choose a unit $ b_\alpha \in R^\times $ such that we have
	$ b_\alpha \equiv \lambda_\alpha \mod \fm $.
	Otherwise, we put $ b_\alpha := 0 \in R $.
	We define
	\[
	g := \sum_{\alpha \in \bN^n } b_\alpha x^\alpha \in R,
	\]
	which is an element in $ R $ since all but finitely many coefficients $ b_\alpha $ are zero.
	By construction, the image of $ g $ in $ \gr_\omega(R) $ is $ G $, 
	i.e., $ \ini_\omega(g) = G $.

	In general, $ g $ is not unique since it depends on a choice of a system of representatives in $ R $ for the residue field $ \bK = R/\fm $.
	If $ \bK \subset R $, then we can uniquely choose $ b_\alpha := \lambda_\alpha $.
	
	\smallskip
	
	Using the above, we can adapt the proofs to obtain Theorems \ref{t1} and \ref{t43}.
	
	\begin{proof}[Proof of Theorem \ref{t1}]
		Let $ w_0 := \omega(f) $. 
		We have $ f|_E = \ini_\omega (f) $. 
		By the hypothesis of the theorem, $ \ini_\omega(f) = G H \in R_{w_0} $. 
		Since all monomials of $ \ini_\omega (f) $ have the same weight,
		it follows $ G\in R_{w}$ and $H\in R_{z}$ 
		for some $w$, $z\in \bN^{n-1}$ such that $ w + z = w_0 $.
		We set $ G_w := G $ and $ H_z := H $.
		
		Let $ f = \sum_\alpha a_\alpha \underline x^\alpha $ be a finite expansion of $ f $ in $ R $.
		Fix $\alpha_0 \in \bN^{n}$ such that $a_{\alpha_0}\neq0$ and 
		$\omega(\underline{x}^{\alpha_0}) = w_0$, 
		i.e., $ \underline{X}^{\alpha_0} $ is a monomial appearing in $ \ini_\omega(f) $. 
		By Lemma~\ref{Lem:edge}, we have $\scalar{\xi}{\alpha} \geq \scalar{\xi}{\alpha_0}$
		for every $\xi\in\mathbb{R}_{\geq0}^n$ which is orthogonal to $E$. 
		This yields $\omega({\underline x}^{\alpha-\alpha_0})\in M$. 
		Since $\omega({\underline x}^{\alpha_0})=w+z$, 
		we obtain $ \omega({\underline x}^{\alpha}) = w + z + i = w_0 + i $, for some $i\in M$. 
		
		We choose a lift $ g_w \in R $ (resp.~$ h_z \in R $) of $ G_w \in R_w $ (resp.~$H_z \in R_z $), as described before.
		We define our first approximation of $ f $ as $ \phi_1 :=  g_w \cdot h_z $.
		Note that $ \ini_\omega(f) = \ini_\omega(\phi_1) $.
		We introduce $ f_1 := f - \phi_1 $.
		By the previous, we have $ w_1 := \omega (f_1) = w_0 + i_1 $, for some $ i_1 \in M $ and $ i_1 \neq ( \underline 0 ) $.
		In particular, $ \ini_\omega(f_1) \in R_{w+z+i_1} \subset \gr_\omega(R) $. 
		By assumption, $ G_w = G $ is not divisible by a monomial, hence, by Lemma \ref{l3}, 
		there are $ H_{z+i_1} \in R_{z+i_1} $ and 
		$  G_{w+i_1} \in R_{w+i_1} $ such that
		$
		\ini_\omega(f_1) = 
		G_w H_{z+i_1} + H_z G_{w+i_1}.
		$
		We choose $ h_{z+i_1}, g_{w+i_1} \in R $, as described before.
		The second approximation $ \phi_2 \in R $ of $ f $ is then defined as 
		\[
		\phi_2 := (g_w + g_{w+i_1})(h_z + h_{z+i_1} )
		= g_w h_z + (g_w h_{z+i_1} + h_z g_{w+i_1}) + g_{w+i_1} h_{z+i_1}.
		\]
		By construction, $ \omega (g_{w+i_1} h_{z+i_1}  ) > w_1 $.
		For
		$
		f_2 := f - \phi_2
		$,
		we obtain that $ w_2 := \omega(f_2) > w_1 $.
		We proceed with the construction and finally get
		$  g := \sum_{i\in M } g_{w+i} $, 
		$  h := \sum_{i\in M } h_{z+i} \in 
		\widehat R 
		$
		with the property $ f = g \cdot  h $ in $ \widehat R $.

		The claim concerning $ E = E_1+ E_2 $ follows with the same argument as in the case $ R = \bK[\![\underline x ]\!] $.
	\end{proof}

	\smallskip 
	
	Finally, the proof of Theorem \ref{t43} for $ R = \bK[\![\underline x ] \!] $ 
	can easily be adapted to the general setting:

	\begin{proof}[Proof of Theorem \ref{t43}]
		Proceeding as in the previous proof, but using Lemma~\ref{l5} instead of Lemma~\ref{l3}
		we obtain $\overline{g},\overline{h}\in \widehat R [\![ y ]\!]$ such that 
		$f=\overline{g}\overline{h}$, $\overline{g}|_{E_1}=G$ and $\overline{h}|_{E_2}=H$ 
		for some segments $E_1,E_2$, where $E=E_1+E_2$. 
		The assumptions that the loose edge $E$ is descendant and the polynomial $G$  
		is monic with respect to  $y$
		imply that the Newton polyhedron of $\overline{g}$ 
		has a~vertex $(0,\dots,0,d)$ for some positive integer $d$.
		Hence, the monomial $ v y^d$ appears in an expansion of $ \overline g $, for some unit $ v \in \widehat R[\![y]\!]^\times $.
		The Weierstrass preparation theorem 
		(\cite[Theorem~9.2]{Lang})
		implies that
		there exist a unit $ u \in \widehat{R}[\![y]\!]^\times $ and 
		a Weierstrass polynomial $ g \in \widehat R [y] $ of degree $d$  such that
		$
		\overline g = u  g.
		$
		We define $ h := u \overline h $ and obtain that
		$ f = g h $.
		As before, $ f, g \in \widehat R [y] $ imply 
		$ h \in \widehat R[y] $.
	\end{proof}

	\smallskip

	\section{Examples}
	Let us discuss some examples.
	
\begin{ex}
	Let $ R $ be a regular local ring (not necessarily complete) of dimension three with regular system of parameters $ (x,y,z) $ and residue field $ \bK $. 
	Suppose that $ R $ contains a field of characteristic different from two.
	(E.g., $ R = \bK[x,y,z]_{\langle x,y,z \rangle} $ and $ \mbox{char}(\bK) \neq 2 $).
	Consider the element
	\[
		f = x^6 y^2 - z^4 + x y z^4 - x^7 y^5 z^2 \in R.
	\]
	It is not hard to see that the Newton polyhedron $ \Delta(f) $ has exactly two vertices, $ v_1 = (6,2,0) $ and $ v_2 = (0,0,4) $.
	Hence, $ v_1 $ and $ v_2 $ determine a loose edge $ E $ of $ \Delta(f) $.
	The symbolic restriction of $ f $ to $ E $ is
	\[
		f|_E = X^6 Y^2  - Z^4 = (X^3 Y - Z^2)(X^3 Y + Z^2) \in \bK[X,Y,Z] \cong \gr_\omega(R),  
	\]
	where $ \omega = \omega(\xi) $, for $ \xi \in \bR_{\geq 0}^3 $ appropriately chosen such that $ \Delta(f)^\xi = E $.
	(Here, and in the examples below, we leave the explicit computation of $ \xi $ and $ \omega $ as an exercise to the reader).
	
	By Theorem~\ref{t1}, the factorization of $ f|_E $ implies that there are $ g,h \in \widehat R $ in the completion such that $ f = g h $. 
	Indeed, if we have a closer look at $ f $, we observe that
	$f = x^6 y^2 ( 1 - x y^3 z^2 ) - z^4 (1 - xy)$,
	which factors in the completion as
	\[
		f = 
		(x^3 y \sqrt{1 - x y^3 z^2 } - z^2 \sqrt{1 - xy})
		(x^3 y \sqrt{1 - x y^3 z^2 } + z^2 \sqrt{1 - xy}) \in \widehat R.
	\]

\end{ex}	

\smallskip 

\begin{ex}
	Let $ R $ be any regular local ring of dimension three with regular system of parameters $ (x,y,z) $ and residue field $ \bK $. 
	Let 
	\[
		f = x y z + x^3 y^3 + x^3 z^3 + y^3 z^3.
	\]
	We observe that the Newton polyhedron is of the form as the one in Figure 3 at the beginning. 
	In particular, it has three loose edges. 
	If we consider the edge $ E $ determined by the vertices $ (1,1,1) $ and $ (3,3,0) $, 
	we see that the corresponding symbolic restriction is 
	$ f|_E = XYZ + X^3Y^3  = XY(Z + X^2 Y^2) \in \bK[X,Y,Z] $.
	Therefore, $ f $ is not irreducible in the completion $ \widehat R $, by Theorem~\ref{t1}.
\end{ex}

Note that in the previous example, $ R $ is allowed to have mixed characteristics,
e.g., we might have $ R = \mathbb{Z}[y,z]_{\langle p,y,z \rangle } $, where $ x=p \in \mathbb{Z} $ is a prime number.

\smallskip 

\begin{ex}
	Let $ R $ be a regular local ring of dimension two with regular system of parameters 
	$ (x_1, x_2) $ and residue field $ \bK $.
	Consider the Weierstrass polynomial
	\[
		f = y^8 + (x_1^3 - x_2^2) y^3 + x_1^5 x_2^4 y^2 - x_1^{15} x_2^{18} \in R[y].
	\]
	In Figure 6 the corresponding Newton polyhedron is drawn.
	\begin{center}
		\begin{tikzpicture}[scale=0.25]
		\draw [->](0,0,0) -- (21,0,0); 
		\draw[->](0,0,0) -- (0,10,0) ; 
		\draw[->](0,0,0) -- (0,0,18);
		
		\filldraw[black] (0,8,0)  coordinate (v0) circle (5pt);
		\filldraw[black] (2,4,0)  coordinate (v1) circle (5pt);
		\filldraw[black] (0,4,3)  coordinate (v2) circle (5pt);
		\filldraw[black] (5,2,4)  coordinate (v3) circle (5pt);
		\filldraw[black] (18,0,15)  coordinate (v4) circle (5pt);
		
		\draw[very thick] (v0) -- (v1) -- (v2) -- (v0);
		\draw[very thick] (v1) -- (v3) -- (v2);	
		\draw[very thick,blue] (v3) -- (v4);	
		
		\draw[thick, dashed]  (0,8,18) -- (0,8,0) -- (21,8,0);
		\draw[thick, dashed]  (0,9.9,0) -- (0,8,0);
		\draw[thick, dashed]  (v2) -- (0,4,18);
		\draw[thick, dashed]  (v1) -- (21,4,0);
		\draw[thick, dashed]  (5,2,18) -- (v3) -- (21,2,4);				
		\draw[thick, dashed]  (18,0,18) -- (v4) -- (21,0,15);
		
		\draw[dotted]  (0,8,18) -- (0,4,18) -- (5,2,18) -- (18,0,18);
		\draw[dotted]  (21,8,0) -- (21,4,0) -- (21,2,4) -- (21,0,15);

		\filldraw[black] (0,8,0) circle (5pt);
		\filldraw[black] (2,4,0) circle (5pt);
		\filldraw[black] (0,4,3) circle (5pt);
		\filldraw[black] (5,2,4) circle (5pt);
		\filldraw[black] (18,0,15) circle (5pt);
		
		\node [below=1.2cm, align=flush center,text width=2cm] at (5.5,-3.5,0)    {Fig.\ 6};
		\end{tikzpicture}
	\end{center}
	The edge $ E $ determined by the vertices $ (5,4,2) $ and $ (15,18,0) $ is loose and descendant.
	Since the symbolic restriction of $ f $ to $ E $ is 
	\[ 
		f|_E = X_1^5 X_2^4 Y^2 - X_1^{15} X_2^{18} = X_1^5 X_2^4 ( Y - X_1^{5} X_2^{7})( Y + X_1^{5} X_2^{7}) \in \bK[X_1, X_2, Y],
	\]  
	Theorem~\ref{t43} implies that $ f $ is not irreducible in $ \widehat R [y] $.
\end{ex}

\begin{ex}
	Let $ R = \mathbb{Z}_p $ be the $ p $-adic integers, for some $ p \in \mathbb{Z} $ prime.
	Let 
	\[
		f = y^3 + 270 y + 540 \in R[y].
	\]
	Note that $ 270 = 2 \cdot 3^3 \cdot 5 $ and $ 540 = 2^2 \cdot 3^3 \cdot 5  $
	and the residue field of $ R $ is $ \mathbb{F}_p $.
	\begin{enumerate}
		\item 
		If $ p = 2 $, then $ f = y^3 + \epsilon_1 p y + \epsilon_3 p^2 $, for $ \epsilon_1, \epsilon_3 \in \mathbb{Z}_2^\times $ units.
		The Newton polyhedron has three vertices $ (0,3), (1,1), (2,0) $.
		If $ E_2 $ denotes the edge given by $ (0,3) $ and $ (1,1) $, then
		$  
			f|_{E_2} = Y^3 + P Y = Y (Y^2 + P) \in \mathbb{F}_2[P,Y] .
		$ 
		Hence, Theorem~\ref{t43} implies that $ f $ factors in $ \widehat{\mathbb{Z}_2}[y] $.
		
		\smallskip 
		
		\item 
		If $ p = 3 $, then $ f = y^3 + \mu_1 p^3 y + \mu_3 p^3 $, for $ \mu_1, \mu_3 \in \mathbb{Z}_3^\times $ units.
		The Newton polyhedron has two vertices $ (0,3) $ and $ (3,0) $.
		Let $ E_3 $ be the corresponding compact edge. 
		We have $ f|_{E_3} = Y^3 + P^3 = (Y + P)^3 \in \mathbb{F}_3[P,Y] $. 
		Therefore, we cannot apply Theorem~\ref{t43}.
	\end{enumerate}
	Note that if $ p = 5 $, then $ f = y^3 + \rho_1 p y + \rho_3 p $, for $ \rho_1, \rho_3 \in \mathbb{Z}_5^\times $ units, which means that $ f \in \mathbb{Z}_5[y] $ is irreducible in $ \widehat{\mathbb{Z}_5}[y] $.
\end{ex}

\smallskip 	
	
	\section{Relation with known results}
	Corollary~\ref{w2}  generalizes to $n>2$ variables the following well-known fact.
	
	\begin{Tw}[{\cite[Lemma 3.15]{hefez}}] 
		Assume that a power series $f\in \bC[\![x,y]\!]$ written as a sum of homogeneous polynomials 
		$f=f_d+f_{d+1}+\dots$, where the subindex means the degree,  is irreducible. 
		Then $f_d$ is a power of a linear form.
	\end{Tw}
	
	\medskip
	Let $\bK$ be a field and fix a weight $\omega(x^iy^j) := ni + mj$ for $n,m\in\bZ_{>0}$. 
	Given $0 \neq F \in\bK[\![x, y]\!]$, we will consider its decomposition in 
	$\omega$-quasihomogeneous forms
	$$ F(x, y) = F_{a+b}(x, y) + F_{a+b+1}(x, y) + \dots\,, $$
	where the subindex means the $\omega$-weight.
	
	\begin{Tw}[{\cite[Lemma A1]{artal2015high}}]
		Assume that $F_{a+b}(x, y) = f_a(x, y)g_b(x, y)$, where
		$f_a$, $g_b \in\bK[x, y]$ are quasihomogeneous and coprime. 
		Then, there exist
		$$f, g \in\bK[\![x,y]\!],\quad f = f_a + f_{a+1} + \dots\,,\quad  g = g_b + g_{b+1} + \dots
		$$
		such that $F = fg$. Moreover if $f_a$ is an irreducible polynomial, 
		then $f$ is an irreducible power series.
	\end{Tw}
	Theorem~\ref{t1} was inspired by the method of the proof of Theorem 5.2.
	
	\medskip
	Corollary~\ref{w2} generalizes the following result of \cite{barroso2005decomposition}.
	
	\begin{Tw}
		If $\phi\in \bC\{x_1,\dots ,x_n\}$ is irreducible and has a polygonal Newton polyhedron $\Delta(\phi)$,
		then the polyhedron $\Delta(\phi)$ has only one compact edge $E$ and the polynomial $\phi|_E$ is a power of an irreducible polynomial. 
	\end{Tw}
	
	We recall that a Newton polyhedron is called polygonal if all its compact edges are loose. 
	
	\medskip
	Theorem \ref{t43} generalizes the main result of \cite{rond2017irreducibility} quoted below.
	For this, let us recall that a monic Weierstrass polynomial $ P(Z) \in k[\![x_1,\dots,x_n]\!][Z] $ of degree $ d $ 
		has an {\em orthant associated polyhedron} 
		if the projection of its Newton polyhedron from the distinguished point 
		$ (0, \ldots, 0, d) \in \bR^{n+1}_{\geq 0 } $ 
		to the subspace corresponding to $ (x_1, \ldots, x_n) $ is a polyhedron with exactly one vertex (\cite[Definiton~2.2]{rond2017irreducibility}).
	
	\begin{Tw}
		Let $P(Z)\in k[\![x_1,\dots,x_n]\!][Z]$ be a monic Weierstrass polynomial. Assume
		that P(Z) has an orthant associated polyhedron and that $P|_{\Gamma}\in k[x_1,\dots,x_n,Z]$ 
		is the product of two coprime monic polynomials 
		$S_1$, $S_2 \in k[x_1,\dots,x_n,Z]$, respectively, of degree $d_1$ and $d_2$. 
		Then there exist two monic polynomials $P_1$, $P_2 \in k[\![x_1,\dots,x_n]\!][Z]$, respectively,
		of degrees $d_1$ and $d_2$ in $Z$ such that \\
		i) $P = P_1P_2$, \\
		ii) there is at least one $i\in \{1, 2\}$ such that $P_i$ has an orthant associated polyhedron
		and if $\Gamma_i$ denotes the compact face of $\Delta(P_i)$ containing the points
		$(0, \dots,0,d_i)$, then $P_i|_{\Gamma_i} = S_i$ and $\Gamma_i$ is parallel to $\Gamma$.
	\end{Tw}
	
	\noindent 
	\textbf{Remark.} Let $f\in k[\![x_1,\dots,x_{n}]\!][y]$ be a Weierstrass polynomial of degree $d$. 
	Then $f$ has an orthant associated polyhedron 
	if and only if $\Delta(f)$ has a loose edge with endpoint $(0,\dots,0,d)$.
	The reason for this is the following:
		$ \Delta(f) $ not having a loose edge with endpoint $ (0,\ldots, 0, d ) $ is equivalent to the existence of a compact face of dimension $ \geq 2 $, where one of the vertices is $ (0, \ldots, 0 , d ) $. 
		The latter is equivalent to the condition that the projection of the Newton polyhedron from $ (0, \ldots, 0, d) $ to the subspace determined by $ (x_1, \ldots, x_n) $ has at least two vertices which correspond to vertices of the considered compact face. 
	We also refer to \cite[Lemma~3.1]{parusinski2012abhyankar}.
	
	
	\section*{Appendix: Alternative proof for Lemma~\ref{Lem:sum}}
	
	\setcounter{section}{2} 
	\setcounter{Tw}{3}

	Recall the notations of section~2. 
	We outline an alternative proof for 
	
	\begin{lem}
		Let $w\in\bN^{n-1}$ and $z\in M$. 
		Assume that~$R_w$ contains two coprime monomials. 
		Then $$\dim R_{w+z}=\dim R_w+ \dim R_z-1.$$\label{l1}
	\end{lem}
	
	\begin{proof}
		For any $v\in\bN^{n-1}$ 
		the dimension of the vector space $R_v$ is equal to the number of 
		monomials ${\underline x}^{\alpha}$ such that $\omega({\underline x}^{\alpha})=v$, 				hence is equal to the number of lattice points, i.e. all points with integer coordinates, in the 			set 
		$$l_v = \{\,\alpha\in \bR_{\geq 0}^n : 
		(\scalar{\xi_1}{\alpha},\dots, \scalar{\xi_{n-1}}{\alpha})=v\,\} . $$
		Notice that $l_v$ is the intersection of the straight line 
		$\{a+tc: t\in \mathbb{R}\}$ with $\bR_{\geq 0}^n$, where $\omega({\underline x}^{a})=v			$ and 
		$c$ is a primitive lattice vector parallel to the edge $E$.

		By the assumption, $R_w$ contains two coprime monomials 
		${\underline x}^{a}$ and ${\underline x}^{b}$.  
		Hence $\min(a_1,b_1)=\dots=\min(a_n,b_n)=0$.
		Thus there exist two nonempty and disjoint subsets $A$, $B$ of $\{1,\dots,n\}$ 
		such that $a_i>0 \Leftrightarrow i\in A$ and $b_i>0 \Leftrightarrow i\in B$.
		Without loss of generality we may assume that the vector $c$ is pointed in the direction    		of $b-a$. Then 
		\begin{equation} \label{E1}
		l_w=\{\, a+tc: t\in [0, m]\,\} 
		\end{equation}
		for some positive integer $m$. 
		
		\medskip
		
		Now we will show that the set $l_z$ is nonempty and give its description.
		
		Let $d\in \bZ^n$ be such that $\omega({\underline x}^d)=z$. 
		Consider the system of inequalities 
		$$ d_j+tc_j\geq 0 \quad \mbox{for $j\in B$}.
		$$ 
		Its solution space is an interval $[t_1,\infty)$. 
		
		Let $a'=d+t_1 c$. Then $a'_j\geq 0$ for all $j\in B$ and $a'_{j_0}=0$ for some $j_0\in B$. 
		
		Consider the vectors 
		$v_j=e_j-\frac{c_j}{c_{j_0}}e_{j_0}$ for $j\in \{1,\dots, n\}\setminus B$, 
		where $e_1$, \dots $e_n$ is the standard basis of $\bR^n$.
		These vectors belong to $\bR_{\geq 0}^n$ and are orthogonal to $E$. 
		By the assumption that $z\in M$ we get 
		$a'_j  = \scalar{v_j}{a'} \geq 0$. It follows that $a'\in \bR_{\geq 0}^n$. 
		
		Using the same argument, but starting from the system of inequalities 
		$$ d_i+tc_i\geq 0 \quad \mbox{for $i\in A$}
		$$ 
		
		we show that there exist $t_2\in\bR$ such that  the point
		$b'=d+t_2c$ belongs to $\bR_{\geq 0}^n$ and $b'_{k}=0$ for some $k\in A$.
		
		It follows from the above that $a',b'\in l_z$ and 
		\begin{equation} \label{E2}
		l_z=\{\, d+tc: t\in [t_1, t_2]\,\} 
		\end{equation}
		for some rational numbers $t_1<t_2$.
		
		Finally, $l_{w+z}$ is a segment with endpoints $a+a'$, $b+b'$ which 
		can be written in the form 
		\begin{equation} \label{E3}
		l_{w+z}=\{\, a+d+tc: t\in [t_1, t_2+m]\,\} .
		\end{equation}
		
		Using (\ref{E1})--(\ref{E3}) and counting 
		integer points in the segments $[0,m]$, $[t_1,t_2]$, $[t_1,t_2+m]$ we get the lemma. 
	\end{proof}

\end{document}